\documentclass[12pt,dvips,a4paper]{article}

\usepackage{graphicx}
\usepackage{amssymb, latexsym, amsfonts, amsmath, lscape, amscd,
color, epsfig}
\usepackage{amsthm}
\usepackage{floatflt}

\newtheorem{theorem}{Theorem}[section]

\newtheorem{lemma}[theorem]{Lemma}
\newtheorem{corollary}[theorem]{Corollary}
\newtheorem{observation}[theorem]{Observation}
\newtheorem{conjecture}[theorem]{Conjecture}

\begin{document}
\title{Signed graphs with two negative edges}

\author{Edita Rollov\'a\thanks{European Centre of Excellence, New Technologies for the Information Society, Faculty of Applied Sciences, University of West Bohemia, Technick\'a 8, 30614, Plze\v n, Czech Republic; rollova@ntis.zcu.cz}, Michael Schubert\thanks{Fellow of the International Graduate School "Dynamic Intelligent Systems"},  Eckhard Steffen\thanks{
		Paderborn Institute for Advanced Studies in
		Computer Science and Engineering,
		Paderborn University,
		Warburger Str. 100,
		33102 Paderborn,
		Germany;			
		mischub@upb.de; es@upb.de}}

\maketitle

\abstract{
The presented paper studies the flow number $F(G,\sigma)$ of flow-admissible signed graphs $(G,\sigma)$ with two negative edges. We restrict our study to cubic graphs, because for each non-cubic signed graph $(G,\sigma)$ there is a set ${\cal G}(G,\sigma)$ of cubic graphs such that  $F(G, \sigma) \leq \min \{F(H,\sigma_H) : (H,\sigma_H) \in {\cal G}(G)\}$. 
We prove that $F(G,\sigma) \leq 6$ if $(G,\sigma)$ contains a bridge and $F(G,\sigma) \leq 7$ in general. We prove better bounds, if there is an element $(H,\sigma_H)$ of ${\cal G}(G,\sigma)$ which satisfies some additional conditions. In
particular, if $H$ is bipartite, then $F(G,\sigma) \leq 4$ and the bound is tight. If $H$ is 3-edge-colorable or critical or if it has a sufficient cyclic edge-connectivity, then $F(G,\sigma) \leq 6$. Furthermore, if Tutte's 5-Flow Conjecture is true, then $(G,\sigma)$ admits a nowhere-zero 6-flow  endowed with some strong properties.}

\section{Introduction}
In 1954 Tutte stated a conjecture that every bridgeless graph admits a nowhere-zero 5-flow (\textit{5-flow conjecture}, see \cite{Tutte_1954}). Naturally, the concept of nowhere-zero flows has been extended in several ways. In this paper we study one generalization of them -- nowhere-zero flows on signed graphs. Signed graphs are graphs with signs on edges. It was conjectured by Bouchet \cite{Bouchet_1983} that signed graphs that admit a nowhere-zero flow have a nowhere-zero 6-flow.
Recently, it was announced by DeVos \cite{DeVos} that such signed graphs admit a nowhere-zero 12-flow, which is the best current general approach to Bouchet's conjecture.

Bouchet's conjecture has been confirmed for particular classes of graphs \cite{MS,MR,SS} and also for signed graphs with restricted edge-connectivity (for example \cite{Raspaud_Zhu_2011}). By Seymour \cite{Sey} it is also true for signed graphs with all edges positive, because they correspond to the unsigned case.

In this paper we study signed graphs with two negative edges. It is the minimum number of negative edges for which Bouchet's conjecture is open, because signed graphs with one negative edge are not flow-admissible. This class of signed graphs is further interesting for its connection with Tutte's 5-flow conjecture.
Suppose there exists $k$ such that every signed graph with $k$ negative edges admits a nowhere-zero 5-flow. Take any bridgeless graph $G$ and identify a vertex of all-positive $G$ with a vertex of a flow-admissible signed graph with $k$ negative edges. The resulting signed graph is flow-admissible with $k$ negative edges, so it admits a nowhere-zero 5-flow, as well as all-positive $G$, and hence also $G$.
Therefore the following holds. 

\begin{observation}
If there exists $k$ such that every flow-admissible signed graph with $k$ negative edges admits a nowhere-zero 5-flow, then Tutte's conjecture is true.
\end{observation}
 
Since for every $k\geq 3$ there is a signed graph with $k$ negative edges which does not admit a nowhere-zero $5$-flow (see \cite{SS}), the class of signed graphs with two negative edges is of a great importance. In the opposite direction we will prove that Tutte's conjecture implies Bouchet's conjecture for signed graphs with two negative edges. 
 
In the next section we introduce necessary notions and provide a couple of well-known results on flows. In Section 3 we show how to deal with small edge-cuts, and finally, in Sections 4-6 we prove results on flows for signed graphs with two negative edges.

\section{Preliminaries}

\textit{A signed graph} $(G,\sigma)$ is a graph $G$ and a function $\sigma:\ E(G) \to \{-1,1\}$. The function $\sigma$ is called \textit{a signature}. The set of edges with
negative signature is denoted by $N_\sigma$. It is called \textit{the set of negative edges}, while $E(G) - N_{\sigma}$ is called \textit{the
set of positive edges}. If all edges of $(G,\sigma)$ are positive, i. e. when $N_\sigma=\emptyset$, then $(G,\sigma)$ will be denoted by $(G,\texttt{1})$ and will be called an \textit{all-positive signed graph}.

An assignment $D$ that assigns a direction to every edge according to a given signature is called an \textit{orientation} of $(G,\sigma)$. A positive edge can be directed like \includegraphics[scale=0.1]{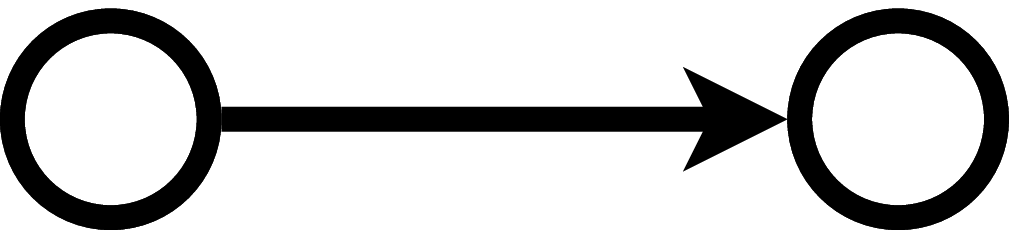} or like \includegraphics[scale=0.1]{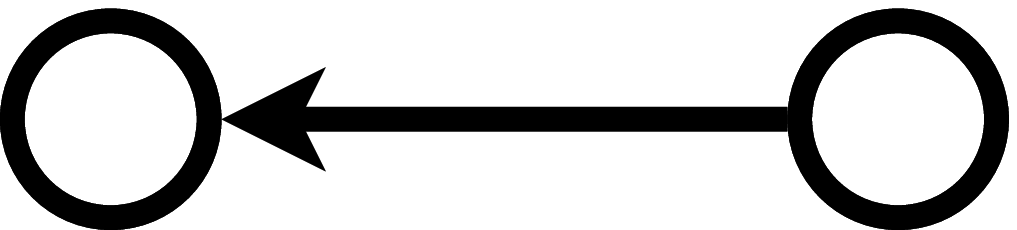}. A
negative edge can be directed like \includegraphics[scale=0.1]{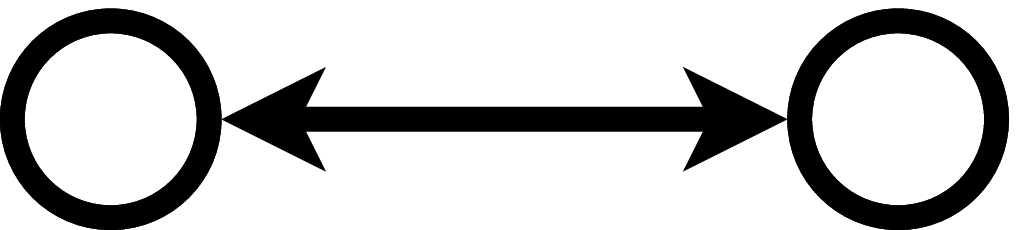} (so-called \textit{extroverted edge}) or like \includegraphics[scale=0.1]{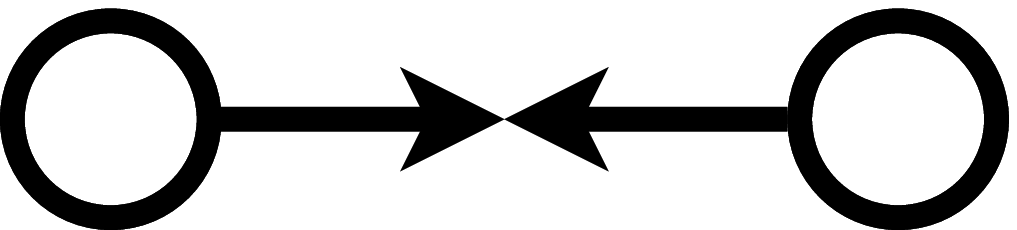} (so-called \textit{introverted edge}). An oriented signed graph is also called \textit{a bidirected graph}.
Sometimes it is helpful to consider an edge $e=vw$ as two half-edges $h_v(e)$ and $h_w(e)$ and the orientation of the edges as an orientation of the half-edges. 

Let $(G,\sigma)$ be a signed graph. \textit{A switching at $v$} defines a graph $(G,\sigma')$ with $\sigma'(e) = -\sigma(e)$ if $e$ is incident to $v$, and $\sigma'(e) = \sigma(e)$ otherwise. We say that signed graphs
$(G,\sigma)$ and $(G,\sigma^*)$ are \textit{equivalent} if they can be obtained from each other by a sequence of switchings. We also say that
$\sigma$ and $\sigma^*$ are \textit{equivalent signatures} of $G$. If we consider a signed graph with an orientation $D$, then switching at $v$ is a change of the orientations of the half-edges that are incident with $v$. If $D^*$ is the resulting orientation, then we say that $D$ and $D^*$ are \textit{equivalent orientations}.

Let $A$ be an abelian group. An \emph{$A$-flow} $(D,\phi)$ on $(G,\sigma)$ consists of an orientation $D$ and a function $\phi : E(G) \rightarrow {A}$ satisfying \textit{Kirchhoff's law}: for every vertex the sum of incoming values equals the sum of outgoing values. If $0\notin \phi(E(G))$, then we say that the $A$-flow is \textit{nowhere-zero}. Let $k$ be a positive integer. A nowhere-zero $\mathbb{Z}$-flow such that  $ -k< \phi(e) <k$ for every $e\in E(G)$ is called a \textit{nowhere-zero $k$-flow}. 
A signed graph $(G,\sigma)$ is \textit{flow-admissible} if it admits a nowhere-zero $k$-flow for some $k$. The flow number of a flow-admissible signed graph $(G,\sigma)$  is
\\
$$F((G,\sigma)) = \min\{ k : (G, \sigma) \textrm{ admits a nowhere-zero } k \textrm{-flow}  \}.$$ 
\\
This minimum always exists and we will abbreviate $F((G,\sigma))$ to $F(G,\sigma)$.

If  $(G,\sigma)$ admits a nowhere-zero $A$-flow $(D,\phi)$ and $(G,\sigma^*)$ is equivalent to $(G,\sigma)$, then there exists an equivalent orientation $D^*$ to $D$ such that $(D^*,\phi)$ is a nowhere-zero $A$-flow on $(G,\sigma^*)$. To find $D^*$ it is enough to switch at the vertices that are switched in order to obtain $\sigma^*$ from $\sigma$. Thus, it is easy to see that $F(G,\sigma) = F(G,\sigma^*)$.

We note that flows on signed graphs that are all-positive are equivalent to flows on graphs (in fact, a nowhere-zero $k$-flow ($A$-flow, respectively) on a graph $G$ can be defined as a nowhere-zero $k$-flow ($A$-flow, respectively) on $(G,\tt{1})$). This allows us to state known results for flows on graphs in terms of flows on signed graphs, and vice-versa. We will freely make a use of this fact. While a graph is flow-admissible if and only if it contains no bridge, the definition of flow-admissibility for signed graphs is more complicated -- it is closely related to the concept of balanced and unbalanced circuits.

A circuit of $(G,\sigma)$ is \textit{balanced} if it contains an even number of negative edges; otherwise it is \textit{unbalanced}. Note that a circuit of $(G,\sigma)$ remains balanced (resp. unbalanced) after switching at any vertex of $(G,\sigma)$. The signed graph
$(G,\sigma)$ is \textit{an unbalanced graph} if it contains an unbalanced circuit; otherwise $(G,\sigma)$ is \textit{a balanced graph}. It is well known (see e.g.~\cite{Raspaud_Zhu_2011}) that $(G,\sigma)$ is balanced if and only if it is equivalent to $(G,\texttt{1})$. A \textit{barbell} of $(G,\sigma)$ is the union of two edge-disjoint unbalanced cycles $C_1$, $C_2$ and a path $P$ satisfying one of the
following properties:
\begin{itemize}
\item $C_1$ and $C_2$ are vertex-disjoint, $P$ is internally
  vertex-disjoint from $C_1\cup C_2$ and shares an endvertex with each
  $C_i$, or
\item $V(C_1)\cap V(C_2)$ consists of a single vertex $w$, and $P$ is
  the trivial path consisting of $w$.
\end{itemize}
 
Balanced circuits and barbells are called \textit{signed circuits}. They are crucial for flow-admissibility of a signed graph.

\begin{lemma}[Lemma 2.4 and Lemma 2.5 in \cite{Bouchet_1983}]\label{flow-admissible}
Let $(G,\sigma)$ be a signed graph. The following statements are equivalent.
\begin{enumerate}
\item $(G,\sigma)$ is not flow-admissible.
\item  $(G,\sigma)$ is equivalent to $(G,\sigma')$ with $|N_{\sigma'}| = 1$ or $G$ has a bridge $b$ such that
a component of $G-b$ is balanced.
\item $(G,\sigma)$ has an edge that is contained neither in a balanced circuit nor in a barbell.
\end{enumerate}
\end{lemma}

When a signed graph has a single negative edge, it is not flow-admissible by the previous lemma. This can be seen also from the fact that the sum of flow values over all negative edges is 0 provided that the negative edges have the same orientation. Therefore, if a flow-admissible signed graph has two negative edges, which is the case considered in this paper, and the negative edges have opposite orientations, then the flow value on the negative edges is the same for any nowhere-zero $k$-flow.

Let $(D,\phi)$ be a nowhere-zero $k$-flow on $(G,\sigma)$. If we reverse the orientation of an edge $e$ (or of the two half-edges, respectively) and replace $\phi(e)$ by $-\phi(e)$, then we obtain another nowhere-zero
$k$-flow $(D^*,\phi^*)$ on $(G,\sigma)$. Hence, if $(G,\sigma)$ is flow-admissible, then it has always a nowhere-zero flow with all flow values positive.

Let $n \geq 1$ and $P = u_0u_1...u_n$ be a path. We say that $P$ is
\textit{a $v$-$w$-path} if $v= u_0$ and $w=u_n$. Let $(G,\sigma)$ be
oriented. If a path $P$ of $G$ does not contain any negative edge and
for every $i \in \{0, \dots ,n-1\}$ the edge $u_iu_{i+1}$ is directed from
$u_i$ to $u_{i+1}$, then we say that $P$ is \textit{a directed
$v$-$w$-path}.

We will frequently make a use of the following well-known lemma.

\begin{lemma}\label{lemma:directeduvpath}
Let $G$ be a graph and $(D,\phi)$ be a nowhere-zero $\mathbb{Z}$-flow on
$(G,\texttt{1})$. If $\phi(e) > 0$ for every $e \in E(G)$, then for any two
vertices $u$,$v$ of $G$ there exists a directed $u$-$v$-path.
\end{lemma}
\begin{proof}
Assume that there are two vertices $u$ and $v$ for which there exists no directed $u$-$v$-path. Let $U$ be the set that consists of $u$ and all vertices $w$ for which there exists a directed $u$-$w$-path. Then $v\in V(G)-U$ and all edges between $U$ and $V(G)-U$ are directed towards $U$. These edges induce an edge-cut for which Kirchhoff's law is false, because $\phi(e)>0$ for every $e$. But then $(D,\phi)$ is not a flow, a contradiction.
\end{proof}

Flows on signed graphs were introduced by Bouchet~\cite{Bouchet_1983}, who stated the following conjecture.

\begin{conjecture} [\cite{Bouchet_1983}] \label{Bouchet_conj}
Let $(G,\sigma)$ be a signed graph. If $(G,\sigma)$ is flow-admissible, then $(G,\sigma)$ admits a nowhere-zero 6-flow.
\end{conjecture}

Seymour's 6-flow theorem for graphs implies Bouchet's conjecture for all-positive signed graphs.

\begin{theorem} [\cite{Sey}] \label{6_Flow}
If $(G,\texttt{1})$ is flow-admissible, then $(G,\texttt{1})$ admits a nowhere-zero 6-flow.
\end{theorem}

Tutte~\cite{Tutte_1954} proved that a graph has a nowhere-zero $k$-flow if and only if it has a nowhere-zero $\mathbb{Z}_k$-flow. This is not true for signed graphs, but in our paper we will apply the following theorem, which is a straightforward corollary of Theorem 3.2 in~\cite{JLPT}. 

\begin{theorem}[\cite{JLPT}]\label{JLPT}
Let $G$ be a 3-edge connected graph and $v\in V(G)$ be of degree~3. If $a,b,c\in \mathbb{Z}_6$ are such that $a+b+c=0$, then $G$ admits a nowhere-zero $\mathbb{Z}_6$-flow such that the edges incident to $v$ receive flow values $a,b,c$.
\end{theorem}
\medskip

\textit{An edge-coloring} of a graph $G$ is to set a color to every edge of $G$ in such a way that two adjacent edges obtain different colors. We say that $G$ is \textit{$c$-edge-colorable} if there exists an edge-coloring of $G$ that uses at most $c$ colors. The smallest number of colors needed to edge-color $G$ is \textit{chromatic index of $G$}. By Vizing's theorem the chromatic index of a cubic graph is either 3 or 4. Bridgeless cubic graphs with chromatic index 4 are also called \textit{snarks}. Tutte \cite{Tutte_1949, Tutte_1954} proved that a cubic graph $G$ is 3-edge-colorable if and only if $G$ (and hence also $(G,\texttt{1})$) admits a nowhere-zero 4-flow, and that $G$ is bipartite if and only if $G$ (and hence $(G,\texttt{1})$) admits a nowhere-zero 3-flow.
We say that a snark $G$ is \textit{critical} if $(G-e,\texttt{1})$ admits a nowhere-zero 4-flow for every edge $e$. Critical snarks were studied for example in \cite{Kochol_2011, Nedela_Skoviera_1996, Steffen_1998}.


\section{Small edge-cuts}

In Section 4 we will show that Bouchet's conjecture holds for signed graphs with two negative edges that contain bridges. Here, we introduce a useful reduction of 2-edge-cuts (different from the one introduced by Bouchet~\cite{Bouchet_1983}). We start with well-known simple observations.

\begin{lemma}\label{lemma:p2edgecut}
Let $(G,\sigma)$ be a signed graph and $X \subseteq E(G)$ be an edge-cut of $G$. If $X=N_{\sigma}$, then $F(G,\sigma)= F(G,\texttt{1})$.
\end{lemma}

\begin{proof}
Let $(W_1,W_2)$ be a partition of $V(G)$ such that $w_1w_2\in X$ if and only if $w_1\in W_1$ and $w_2\in W_2$.
Switching at all vertices of $W_1$ results in an all-positive signed graph, and hence $F(G,\sigma)=F(G,\texttt{1})$.
\end{proof}

Similarly, we can proof the following lemma.

\begin{lemma}\label{lemma:noflow}
Let $(G,\sigma)$ be a signed graph such that all negative edges $N_{\sigma}$ belong to an $(|N_{\sigma}|+1)$-edge-cut. Then $(G,\sigma)$ is not flow-admissible.
\end{lemma}

\begin{proof}
Let $e$ be the positive edge in the $(|N_{\sigma}|+1)$-edge-cut containing all negative edges of the graph. Note that there exists a switching of $(G,\sigma)$ such that the resulting signature contains only one negative edge, namely $e$. Then by Lemma~\ref{flow-admissible}, $(G,\sigma)$ admits no nowhere-zero $k$-flow.
\end{proof}

In our paper we will make a use of this straightforward corollary.

\begin{corollary}\label{cor:3edgecut}
Let $(G,\sigma)$ be a signed graph such that $|N_{\sigma}|=2$. If $(G,\sigma)$ is flow-admissible, then the two negative edges of $(G,\sigma)$ do not belong to any $3$-edge-cut.
\end{corollary}

Let $X=\{uv,xy\}$ be a 2-edge-cut of $(G,\sigma)$ such that $(G-X,\sigma|_{G-X})$ contains a component that is all-positive. If $X=N_{\sigma}$, then $F(G,\sigma)\leq 6$ by Lemma~\ref{lemma:p2edgecut} and Theorem~\ref{6_Flow}. If $uv$ is positive, then we will use the following reduction. A \textit{2-edge-cut reduction of $(G,\sigma)$} with respect to the edge-cut $\{uv,xy\}$ is a disjoint union of two signed graphs, $(G_1,\sigma_{G_1})$ and all-positive $(G_2,\sigma_{G_2})$, that are obtained from $(G,\sigma)$ as follows: remove $uv$ and $xy$ and add a positive edge $vy$ and an edge $ux$ whose sign equals $\sigma(xy)$. Note that $|N_{\sigma_{G_1}}| = |N_{\sigma}|$.
We say that $(G_1,\sigma_{G_1})$ and $(G_2,\sigma_{G_2})$ are \textit{resulting graphs} of the 2-edge-cut reduction of $(G,\sigma)$ (with respect to a 2-edge-cut $\{uv,xy\}$).  

\begin{observation}\label{obser:2red}
The resulting graphs of a 2-edge-cut reduction of a flow-admissible signed graph are flow-admissible.
\end{observation}

\begin{proof}
Let $(G_1,\sigma_{G_1})$ and all-positive $(G_2,\sigma_{G_2})$ be the resulting signed graphs obtained from the 2-edge-cut reduction of $(G,\sigma)$ with respect to $\{uv,xy\}$. Suppose first that $(G_2,\sigma_{G_2})$ is not flow-admissible. Then $(G_2,\sigma_{G_2})$ contains a bridge, which is also a bridge of $(G,\sigma)$ whose removal yields to an all-positive component of $(G,\sigma)$.  By Lemma~\ref{flow-admissible} $(G,\sigma)$ is not flow-admissible, a contradiction. 

Let $e$ be an edge of $(G_1,\sigma_{G_1})$. By Lemma~\ref{flow-admissible} it is enough to show that $e$ belongs to a signed circuit $C_1$ of $(G_1,\sigma_{G_1})$. Since $(G,\sigma)$ is flow-admissible, there exists a signed circuit $C$ of  $(G,\sigma)$ containing $e$. If $E(C)\subseteq E(G_1)$, then $C_1:=C$, and we are done. Otherwise, $C$ contains at least one of $\{uv,xy\}$. Since all negative edges belong to the component of $G-\{uv,xy\}$ containing $e$, $C$ must contain both of $\{uv,xy\}$. Therefore $C$ contains a $u$-$x$-path $P$ such that $E(P)\cap E(G_1)=\emptyset$. Moreover, $xy$ has the same sign as $ux$ and $P-xy$ is all-positive. Therefore replacing $P$ by $ux$ in $C$ is the desired signed circuit $C_1$ of $(G_1,\sigma_{G_1})$.
\end{proof}

\begin{lemma}\label{lemma:n2edgecut} Let $(G_1,\sigma_{G_1})$ and $(G_2,\sigma_{G_2})$ be the resulting graphs of the 2-edge-cut reduction of $(G,\sigma)$ with respect to a 2-edge-cut $\{uv, xy\}$. Let $k>0$ be integer, and let, for $i=1,2$, $(G_i,\sigma_{G_i})$ admit a nowhere-zero $k$-flow $(D_i,\phi_i)$ such that $ux$ is oriented from $u$ to $x$ and $h_v(vy)$ is oriented towards $v$. If $\phi_1(ux)=\phi_2(vy)$, then $F(G,\sigma)\leq k$.
\end{lemma}

\begin{proof}
We will define a flow on $(G,\sigma)$ directly. Let $D$ be an orientation of the edges of $(G,\sigma)$ such that $D(e)=D_i(e)$ for every edge $e\in E(G_i)\cap E(G)$. Let $uv$ be oriented from $u$ to $v$, $h_x(xy)$ be oriented towards $x$, and $h_y(xy)$ be oriented towards $y$ if and only if $h_y(vy)$ is oriented towards $y$.
 We define $\phi$ as follows: $\phi(e)=\phi_i(e)$ for every $e\in E(G_i)\cap E(G)$ and $\phi(uv)=\phi(xy)=\phi_1(ux)$. Clearly, $(D,f)$ is a nowhere-zero $k$-flow of $(G,\sigma)$.
\end{proof}

For a signed graph $(G,\sigma)$ with two negative edges we say that an all-positive 2-edge-cut $X$ \emph{separates the negative edges} if the negative edges belong to different components of $G-X$. We note that we will not use an equivalent of a 2-edge-cut reduction for 2-edge-cuts that separate negative edges, because the resulting signed graphs may not be flow-admissible. 

An idea to reduce non-separating cuts of size less than 3 appeared first in Bouchet's work (see Proposition 4.2.~in~\cite{Bouchet_1983}). However, his reduction uses contraction of a positive edge, which cannot be used in our paper -- contraction of an edge of a signed graph from a particular class (e.g. bipartite) may result in a signed graph that does not belong to the same class.

\section{Nowhere-zero 4-flows}

The following lemma was proven by Sch\"onberger \cite{Sch}.

\begin{lemma}[\cite{Sch}]\label{lemma:matchingwithe}
If $G$ is a bridgeless cubic graph and $e$ is an edge of $G$, then $G$ has a $1$-factor that contains $e$.
\end{lemma}

\begin{lemma}\label{lemma:1factor}
Let $G$ be a cubic bipartite graph, and let $e,f \in E(G)$.  If any 3-edge-cut contains at most one edge of $\{e,f\}$, then there exists a 1-factor of $G$ that contains both, $e$ and $f$.
\end{lemma}

\begin{proof}
Let $U$ and $V$ be the partite sets of $G$. Let $e=uv$ and $f=xy$ be two edges of $G$ such that $u,x\in U$, $v,y\in V$. If $e$ and $f$ are adjacent, they belong to a (trivial) 3-edge-cut of $G$, a contradiction. Hence $e$ and $f$ are non-adjacent.

If $e$ and $f$ form a 2-edge-cut, then they must belong to the same color class of a 3-edge-coloring of $G$ and hence, there is a 1-factor that contains $e$ and $f$.
In what follows, we assume that $\{e,f\}$ is not a 2-edge-cut.

Let $G'$ be the graph that is constructed from $G - \{e,f\}$ by adding new edges $e'=ux$ and $f'=vy$. Then $G'$ is cubic and bridgeless, because $e$ and $f$ do not belong to any 3-edge-cut of $G$. Thus, by Lemma~\ref{lemma:matchingwithe} there exists a 1-factor $F'$ of $G'$ containing $e'$. We claim that $F'$ contains $f'$. Suppose to the contrary that $f'\notin F'$. Then there exist $v'$ and $y'$ from $U$ such that $vv'$ and $yy'$ are in $F'$.  {The} graph $G'-\{u, x, v, v', y, y'\}$ is bipartite with partite sets of sizes $|U|-4$ and $|V|-2=|U|-2$. Note that such a graph does not have any $1$-factor, which is a contradiction with existence of $F'$. Thus $f'$ must belong to $F'$. In that case $F=F'\cup\{e,f\}-\{e',f'\}$ is a 1-factor of $G$ that contains $e$ and $f$.
\end{proof}

\begin{lemma}\label{prop:1factorimplies4flow}
Let $(G,\sigma)$ be a signed cubic graph with $N_{\sigma}=\{n_1,n_2\}$. If $G$ has a 3-edge-coloring such that
$n_1$ and $n_2$ belong to the same color class, then $(G,\sigma)$ admits a nowhere-zero $4$-flow $(D,\phi)$ such that $\phi(n_1)=\phi(n_2)=2$.
\end{lemma}

\begin{proof} Let $c:\ E(G) \to \{c_1,c_2,c_3\}$ be a $3$-edge-coloring such that $c(n_1)=c(n_2)=c_2$.
It is well known and easy to see that $(G,\texttt{1})$ has a nowhere-zero 4-flow $(D,\phi)$ such that $\phi(x) > 0$ for every $x \in E(G)$ and $\phi(y) = 2$ if $y \in c^{-1}(c_2)$.
Let $n_1 = u_1u_2$ and $n_2 = v_1v_2$ be directed towards $u_2$ and towards $v_2$, respectively. By Lemma~\ref{lemma:directeduvpath}, there is a directed path $P$ from $v_2$ to $u_1$. Moreover, $P$ contains neither $n_1$ nor $n_2$. To obtain an orientation $D'$ of $(G,\sigma)$ reverse the orientation of the half-edges $h_{u_1}(n_1)$ and $h_{v_2}(n_2)$ and the edges of $P$, and
 leave the orientation of all other (half-)edges unchanged. Let $\phi'(x) = 4- \phi(x)$ if $x \in E(P)$, and $\phi'(x) = \phi(x)$ otherwise. It is easy to check that $(D',\phi')$ is a desired nowhere-zero 4-flow on $(G,\sigma)$.
\end{proof}

\begin{theorem}\label{thm:cubbip}
Let $(G,\sigma)$ be a flow-admissible signed cubic graph with $|N_{\sigma}|=2$. If $G$ is bipartite, then $F(G,\sigma)\leq 4$.
\end{theorem}

\begin{proof}
Let $N_{\sigma}=\{n_1, n_2\}$. Since $(G,\sigma)$ is flow-admissible, $n_1$ and $n_2$ do not belong to any $3$-edge-cut by Corollary~\ref{cor:3edgecut}. Thus by Lemma~\ref{lemma:1factor}, $G$ has a 1-factor containing $n_1$ and $n_2$. By Lemma~\ref{prop:1factorimplies4flow}, $F(G,\sigma)\leq 4$.
\end{proof}

The bound given in Theorem~\ref{thm:cubbip} is tight. It is achieved for example on $(K_{3,3},\sigma)$, where the two negative edges are independent (see \cite{MR}).
It is not possible to extend the result of Theorem~\ref{thm:cubbip} to cubic bipartite graphs with any number of negative edges. For example, a circuit of length $6$, where every second edge is doubled and one of the parallel edges is negative for every pair of parallel edges and all the other edges are positive has flow number $6$ (see \cite{SS}).

We would like to note that a choice of flow value on negative edges is important. The signed graph of Figure~\ref{fig} is an example of a signed graph that does not admit a nowhere-zero 4-flow that assigns 1 to negative edges even though it admits a nowhere-zero 4-flow according to Theorem~\ref{thm:cubbip}.

\begin{figure}[h]
\begin{center}
\includegraphics[width=6cm]{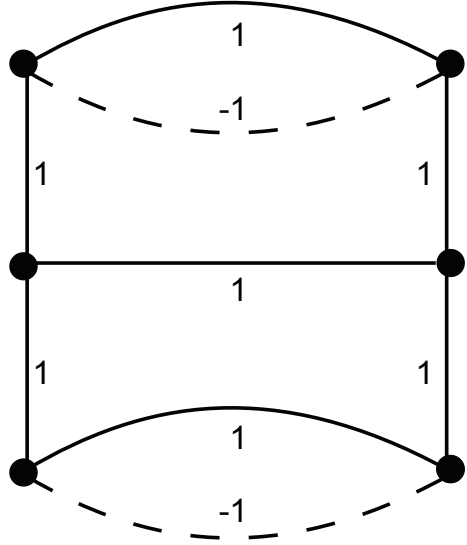}
\caption{{A signed graph for which a choice of flow value on negative edges is important}}\label{fig}
\end{center}
\end{figure}

\section{Nowhere-zero 6-flows}

In this section we prove that Bouchet's conjecture is true for signed graphs with two negative edges where the underlying graph has additional properties. Our first result is on the graphs with bridges, for which we need the following lemma. 

Let $D$ be an orientation of a graph $G$ and $\phi:\ E(G) \to A$ be a function to an abelian group $A$. We say that an \emph{outflow} at a vertex $v$ of $G$ with respect to $(D,\phi)$ is $\sum_{e\in\delta^+(v)} \phi(e)- \sum_{e\in\delta^-(v)} \phi(e)$, where $\delta^+(v)$ ($\delta^-(v)$, respectively) is the \emph{set of outgoing edges} (\emph{incoming edges}, respectively) incident to $v$. 

\begin{lemma}\label{magictrick}
Let $G$ be a graph and let $v$ be a vertex of $G$ of degree 3 incident to $e_1,e_2,e_3$. Let $D$ be an orientation of $G$ such that either $\delta^+(v)=\{e_1,e_2,e_3\}$ or $\delta^-(v)=\{e_1,e_2,e_3\}$.  
If $G$ admits a nowhere-zero $\mathbb{Z}_6$-flow $(D,\phi)$ such that $\phi(e_1)=1$, $\phi(e_2)=x$ and $\phi(e_3)=-1-x$ (for $1\leq x \leq 4$), then $G$ admits a nowhere-zero $6$-flow $(D,\phi')$ such that $\phi'(e_1)=1$, $\phi'(e_2)=x$ and $\phi'(e_3)=-1-x$.
\end{lemma}

\begin{proof}
Let $(D,\phi)$ be an all-positive nowhere-zero $\mathbb{Z}_6$-flow on $G$ such that $\phi(e_1)=1$, $\phi(e_2)=x$ and $\phi(e_3)=1+x$ for a fixed $x\in\{1,2,3,4\}$. If $(D,\phi)$ is also a nowhere-zero $6$-flow, then we are done.

Otherwise $(D,\phi)$ is a nowhere-zero integer function such that the outflow at the vertices of $G$ is a multiple of 6. (Note that the total outflow in $\mathbb{Z}$ taken over all vertices of $G$ is 0, because $(D,\phi)$ is a nowhere-zero $\mathbb{Z}_6$-flow on $G$.) Since $(D,\phi)$ is not an integer flow, there are at least two vertices with non-zero outflow (taken~in~$\mathbb{Z}$). 

Let $w_1$ be a vertex with a positive outflow. We claim that there exists a vertex $w_2$ with a negative outflow such that there is a directed $w_1$-$w_2$-path not containing $e_1$. Suppose the opposite and let $W$ be a subset of $V(G)$ that contains $w_1$ and every vertex $w$ for which there is a directed $w_1$-$w$-path not containing $e_1$. Since $W$ does not contain any vertex with negative outflow, $V(G)-W$ is non-empty. Every edge between $W$ and $V(G)-W$ is oriented towards $W$ except, possibly, the edge $e_1$. By Kirchhoff's law, the total outflow from $V(G)-W$ must be negative, which is possible only when $e_1$ is the only edge between $W$ and $V(G)-W$, because $\phi$ is all-positive. But then $e_1$ is a bridge of a flow-admissible graph $G$, which is a contradiction. Therefore, there is a directed $w_1$-$w_2$-path $P$. To obtain a new nowhere-zero function $(D^*,\phi^*)$, reverse the orientation of the edges of $P$, leave the orientation of all other edges unchanged, and define $\phi^*(f)=\phi(f)$ for $f\notin P$, and  $\phi^*(f)=6-\phi(f)$ for $f\in P$. 
Note that $(D^*,\phi^*)$ is positive on every edge of $G$, and since $e_1\notin P$, $\phi^*(e_1)=1$. We iterate this process until the outflow at every vertex of $G$ is 0. (Note that this process is finite, because the sum of absolute values of the outflows in $\mathbb{Z}$ over all vertices decreases.) 

Let $(D^\#,\phi^\#)$ be the final nowhere-zero function. Since the outflow at every vertex is 0, $(D^\#,\phi^\#)$ is a nowhere-zero $6$-flow (which is also positive on every edge). If $\phi^\#(e_1)=1$, $\phi^\#(e_2)=x$, and $\phi^\#(e_3)=1+x$, then we are done. 
Otherwise $\phi^\#(e_1)=1$, $\phi^\#(e_2)=6-x$, and $\phi^\#(e_3)=5-x$. By Lemma~\ref{lemma:directeduvpath}, there is a directed $u_2$-$u_3$-path $Q$ in $G$, where $u_i\in e_i$, for $i=2,3$. Then $Q\cup e_3 \cup e_2$ is a directed cycle. Reversing the orientation $D^\#$ on $E(Q\cup e_3 \cup e_2)$ and replacing $\phi^\#(e)$ with $6-\phi^\#(e)$
on every edge $e\in Q\cup e_3 \cup e_2$ provides a desired nowhere-zero $6$-flow on~$G$.
\end{proof}

\begin{corollary} \label{6_Flow_1}
Let $G$ be a cubic graph and $f \in E(G)$. If $G$ is bridgeless, then $(G,\texttt{1})$ has a nowhere-zero 6-flow $(D,\phi)$, and we can choose the flow value $\phi(f)$.
\end{corollary}

\begin{proof}
Note that by Theorem~\ref{6_Flow}, $(G,\texttt{1})$ admits a nowhere-zero $6$-flow. If $G$ is 3-edge-connected, then the result follows from Theorem~\ref{JLPT} and Lemma~\ref{magictrick}. Suppose now that $G$ is 2-edge-connected. We show that we can choose the flow value on $f$. Suppose the contrary, and let $G$ be a counterexample  with minimum number of edges. Let $X$ be a 2-edge-cut of $G$. Let $(G_1,\tt{1})$ and $(G_2,\tt{1})$ be the resulting graphs of the 2-edge-cut reduction of $(G,\tt{1})$ with respect to $X$. By Observation~\ref{obser:2red}, $(G_1,\tt{1})$ and $(G_2,\tt{1})$ are flow-admissible. One of $(G_1,\tt{1})$ and $(G_2,\tt{1})$, say $(G_1,\tt{1})$, contains $f$. Since $(G_1,\tt{1})$ is smaller than $G$, it admits a nowhere-zero $6$-flow $(D_1,\phi_1)$ such that we can choose $\phi_1(f)$. Since $(G_2,\tt{1})$ is also smaller than $G$, $(G_2,\tt{1})$ admits a nowhere-zero $6$-flow $(D_2,\phi_2)$ such that $\phi_2(e_2)=\phi_1(e_1)$ where $e_i\in E(G_i)-E(G)$ for $i=1,2$. By Lemma~\ref{lemma:n2edgecut} we can combine $(D_1,\phi_1)$ and $(D_2,\phi_2)$ to obtain a desired nowhere-zero $6$-flow on $G$, which is a contradiction. 
\end{proof}

\begin{theorem}\label{thm:bridges}
Let $(G,\sigma)$ be a flow-admissible signed cubic graph with $N_{\sigma}=\{n_1,n_2\}$. If $(G,\sigma)$ contains a bridge, then $(G,\sigma)$ admits a nowhere-zero 6-flow $(D,\phi)$ such that $\phi(n_1)=\phi(n_2)=1$. 
\end{theorem}

\begin{proof}
Suppose the contrary and let $(G,\sigma)$ be a minimal counterexample in terms of number of edges. 
Let, first, $(G,\sigma)$ contain a 2-edge-cut $X$ that does not separate the negative edges $n_1$ and $n_2$. Suppose that $(G_1,\sigma_{G_1})$ and $(G_2,\sigma_{G_2})$ are resulting graphs of the 2-edge-cut reduction of $(G,\sigma)$ with respect to $X$. By Observation~\ref{obser:2red}, $(G_i,\sigma_{G_i})$ is flow-admissible, for $i=1,2$. Since $(G_1,\sigma_{G_1})$ contains two negative edges and is smaller than $(G,\sigma)$, it admits a nowhere-zero $6$-flow $(D_1,\phi_1)$ such that $\phi_1(n_1)=\phi_1(n_2)=1$. By Corollary~\ref{6_Flow_1}, $(G_2,\sigma_{G_2})$  admits a nowhere-zero 6-flow $(D_2,\phi_2)$ such that $\phi_2(f_2)=\phi_1(f_1)$, where $f_i\in E(G_i)-E(G)$. Finally, by Lemma~\ref{lemma:n2edgecut}, $(G,\sigma)$ admits a nowhere-zero $6$-flow such that the negative edges of $(G,\sigma)$ receive flow value 1. This is a contradiction, and we may assume that every 2-edge-cut of $(G,\sigma)$ separates the negative edges $n_1$ and $n_2$. 

Let $b_1,\ldots,b_l$ be all the bridges of $(G,\sigma)$, for $l\geq1$. Note that neither $n_1$ nor $n_2$ is a bridge, otherwise $(G,\sigma)$ is not flow-admissible. Moreover, since $(G,\sigma)$ is flow-admissible, $b_1,\ldots,b_l$ lie on the same path. Let $(G_0,\sigma_0), \ldots, (G_l,\sigma_l)$ be 2-edge-connected components of $(G,\sigma)-\{b_1,\ldots,b_l\}$ such that $b_i$ is incident to $(G_{i-1},\sigma_{i-1})$ and $(G_i,\sigma_i)$, for $i=1,\ldots, l$. Then $n_1\in E(G_0)$ and $n_2\in E(G_l)$ (or vice versa), otherwise the bridges of $(G,\sigma)$ do not belong to a signed circuit (which is a contradiction with flow-admissibility of $(G,\sigma)$). 

Let $n_1=u_0v_0$ and $n_2=u_lv_l$. For $i\in\{0,l\}$, let $G_i^*$ be an underlying graph obtained from a signed graph $(G_i,\sigma_i)$ by removing $u_iv_i$ and connecting three degree 2 vertices ($u_i$, $v_i$ and an end-vertex of a bridge) into a new vertex $w_i$. We claim that $G_i^*$ is 3-edge-connected. It is easy to see that $G_i^*$ is connected and does not have a bridge, because it is obtained from a 2-edge-connected graph $(G_i,\sigma_i)$ where deleted edge $u_iv_i$ is replaced by a path $u_iw_iv_i$. Suppose for the contrary that $X\subseteq E(G_i^*)$ is a 2-edge-cut of $G_i^*$. If the three neighbors of $w_i$ belong to one component of $G_i^*-X$, then $X$ is a non-separating 2-edge-cut of $(G,\sigma)$, a contradiction. Therefore, there is one component of $G_i^*-X$ containing exactly one neighbor of $w_i$. But then $X$ contains either two edges or exactly one edge incident to $w_i$. In the former case, $G_i^*-w_i=G_i-u_iv_i$ is disconnected, which is impossible, since $(G_i,\sigma_i)$ is 2-edge-connected. In the latter case, $G_i^*-w_i=G_i-u_iv_i$ contains a bridge. This is possible if and only if $u_iv_i$ belongs to a 2-edge-cut of $(G_i,\sigma_i)$, which is a non-separating 2-edge-cut of $(G,\sigma)$, because it contains $u_iv_i$. This is a contradiction, and we conclude that $G_i^*$ is 3-edge-connected. 

By Theorem~\ref{JLPT}, $G_i^*$ admits a nowhere-zero $\mathbb{Z}_6$-flow $(D_i^*,\phi^*_i)$ such that the flow values on edges incident to $w_i$ are $a,b$ and $c$, for $a+b+c=0$. For $G_0^*$, let $a=b=1$ and $c=-2$, where $\phi^*_0(w_0u_0)=\phi^*_0(w_0v_0)=1$. For $G_l^*$, let $a=b=-1$ and $c=2$, where $\phi^*_l(w_lu_l)=\phi^*_l(w_lv_l)=-1$. By Lemma~\ref{magictrick}, $G_i^*$ admits a nowhere-zero $6$-flow $(D_i,\phi_i)$ such that $\phi_i(e)=\phi^*_i(e)$, for every edge $e$ incident to $w_i$. 

Suppose first that $l=1$. We define $(D,\phi)$ on $(G,\sigma)$ as follows. Let $(D,\phi)=(D_i,\phi_i)$ for every edge $e\in E(G)\cap E(G_i)$. Let $n_0$ be extroverted, $n_l$ be introverted and let $\phi(n_0)=\phi(n_l)=1$. Finally, let $b_1$ be oriented from a vertex of $G_0$ to a vertex of $G_l$ and let $\phi(b_1)=2$. It is easy to see that $(D,\phi)$ is a desired nowhere-zero 6-flow on $(G,\sigma)$, a contradiction.

Finally, suppose that $l\geq 2$. Then $(G_j,\sigma_j)$ are all-positive, for $j=1,\ldots,l-1$. Add a new edge $e_j$ to $G_j$ to connect the vertices of degree 2 (there are two such vertices-- the end-vertices of $b_{j}$ and $b_{j+1}$ in $(G,\sigma)$). By Corollary~\ref{6_Flow_1}, $G_j\cup e_j$ admits a nowhere-zero $6$-flow $(D_j,\phi_j)$ such that $\phi_j(e_j)=2$ where $e_j$ is oriented from the end-vertex of $b_{j+1}$ to the end-vertex of $b_j$. We are ready to define $(D,\phi)$ on $(G,\sigma)$. For $i\in\{0,\ldots,l\}$, let $(D|_{E(G_i)},\phi|_{E(G_i)})=(D_i|_{E(G_i)},\phi_i|_{E(G_i)})$, let $n_1$ be extroverted, $n_2$ be introverted and $\phi(n_1)=\phi(n_2)=1$. Finally, for $j=1,\ldots, l$, let $b_j$ be oriented from the vertex of $G_{j-1}$ to the vertex of $G_j$ with $\phi(b_j)=2$. It is easy to see that $(D,\phi)$ is a desired nowhere-zero 6-flow on $(G,\sigma)$, which is a contradiction and end of the proof. 
\end{proof}

In the following we focus on $(G,\sigma)$ where $G$ is 3-edge-colorable or critical. 

\begin{lemma} \label{flow_value_1} Let $G$ be a cubic graph and $e_1,e_2 \in E(G)$. If $G$ is 3-edge-colorable, then $(G,\texttt{1})$ has a nowhere-zero 4-flow $(D,\phi)$ such that $\phi(f)>0$ for every $f\in E(G)$, and $\phi(e_1) = \phi(e_2) = 1$.
\end{lemma}

\begin{proof}
Let $c:\ E(G) \to \{c_1,c_2,c_3\}$ be a $3$-edge-coloring, and let $c(e_1)=c_1$  {and} $c(e_2) \in \{c_1,c_2\}$.
Let $(D_1,\phi_1)$ be a nowhere-zero 2-flow on $c^{-1}(c_1)\cup c^{-1}(c_2)$ and
$(D_2,\phi_2)$ be a nowhere-zero 2-flow on $c^{-1}(c_2) \cup c^{-1}(c_3)$.

In both cases for $c(e_2)$, $(D,\phi)$ is obtained as a combination of $(D_1,\phi)$ and $(D_2,2\phi_2)$. Note that if $c(e_2) = c_2$, then the orientation $D_2$ should be chosen in such a way that
$D_1$ and $D_2$ give opposite directions to $e_2$. The desired flow on $(G,\sigma)$ is obtained from $(D,\phi)$ by reversing each edge with negative value. 
\end{proof}

\begin{theorem}\label{thm:cubic3edgecolourable}
Let $(G,\sigma)$ be a flow-admissible signed cubic graph with $N_{\sigma}=\{n_1,n_2\}$. If $G$ is $3$-edge-colorable or critical, then $(G,\sigma)$ has a nowhere-zero $6$-flow $(D,\phi)$ such that $\phi(n_1)=\phi(n_2)=1$.
\end{theorem}

\begin{proof} 
Let $(G,\sigma)$ be a minimal counterexample to the theorem in terms of number of edges. 
By Lemma~\ref{lemma:p2edgecut} and Theorem~\ref{6_Flow}, $(G,\sigma)$ has no 2-edge-cut containing both negative edges. If $(G,\sigma)$ has a 2-edge-cut containing exactly one negative edge, then apply the 2-edge-cut-reduction. Then a combination of Observation~\ref{obser:2red}, induction, Lemma~\ref{lemma:n2edgecut} and Corollary~\ref{6_Flow_1} yields to a contradiction.
Hence, in the following we assume that no 2-edge-cut of $(G,\sigma)$ contains a negative edge and, by Corollary~\ref{cor:3edgecut}, no 3-edge-cut of $G$ 
contains both negative edges.
\smallskip 
\\
\textit{Case 1: $G$ is 3-edge-colorable.} By Lemma \ref{flow_value_1}, there is a nowhere-zero 4-flow $(D',\phi')$ on $(G,\texttt{1})$ such that $\phi'(n_1) = \phi'(n_2) = 1$, and $\phi'(e)>0$, for every $e\in E(G)$.

Suppose, without loss of generality, that $D'$ orients the edges $n_1 = x_1x_2$ and $n_2 = y_1y_2$ from $x_1$ to $x_2$ and from $y_1$ to $y_2$, respectively. We now define an $x_1$-$y_2$-path $P$ such that $E(P)\cap N_{\sigma}=\emptyset$. If there is a directed $x_1$-$y_2$-path $P_1$ such that $E(P)\cap N_{\sigma}=\emptyset$, we set $P=P_1$. Otherwise, by Lemma~\ref{lemma:directeduvpath}, every directed $x_1$-$y_2$-path contains an edge of $N_{\sigma}$. This is possible if and only if $n_1$ and $n_2$ belong to a same 4-edge-cut of $G$, because $\phi'(n_1) = \phi'(n_2) = 1$ and no 2-edge-cut of $(G,\sigma)$ contains a negative edge and no 3-edge-cut of $(G,\sigma)$ contains both negative edges. Let $f=z_1z_2$ be another edge of the 4-edge-cut, and suppose that $D'$ orients $f$ from $z_1$ to $z_2$ (note that by Kirchhoff's law $\phi'(f)=1$). Then we set $P=P_2\cup f\cup P_3$, where $P_2$ is a directed $x_1$-$z_2$-path such that $E(P_2)\cap N_{\sigma}=\emptyset$ and $P_3$ is a directed $z_1$-$y_2$-path such that $E(P_3)\cap N_{\sigma}=\emptyset$. We are ready to define $(D,\phi)$ on $(G,\sigma)$. Obtain $D$ by reversing the orientation of $h_{x_1}(n_1)$ and $h_{y_2}(n_2)$ and by setting $D(h)=D'(h)$ for every other half-edge $h$ of $(G,\sigma)$. The desired nowhere-zero 6-flow on $(G,\sigma)$ is $(D,\phi)$ with $\phi(e) = \phi'(e) +2$ if $e \in E(P)-f$, $\phi(f)=-1$, and $\phi(e) = \phi'(e)$ otherwise.
\smallskip 
\\
\textit{Case 2: $G$ is critical.}
Suppress $x_1$ and $x_2$ in $G-n_1$ to obtain a 3-edge-colorable cubic graph $G'$.
By Lemma \ref{flow_value_1}, $(G',\texttt{1})$ admits a nowhere-zero 4-flow $(D',\phi')$ such that $\phi'(n_2)=1$. Let $n_2 = y_1y_2$ be directed from $y_1$ to $y_2$.
Clearly, $(D',\phi')$ can be considered also as a nowhere-zero 4-flow on $(G-n_1, \texttt{1})$. Consider a directed $x_1$-$y_2$-path $P_1$ and a  directed $x_2$-$y_2$-path $P_2$
in $(G-n_1, \texttt{1})$. Since $\phi'(n_2) = 1$ and $n_2$ does not belong to any 2-edge-cut, we may assume that $n_2\notin E(P_1)\cup E(P_2)$. Obtain an orientation $D$ of $(G,\sigma)$ by letting $n_1$ be extroverted, reversing 
the orientation of $h_{y_2}(n_2)$, and by setting $D(h)=D'(h)$ for every other half-edge $h$ of $(G,\sigma)$. Let $\phi''(e) = \phi'(e) + 1$ if $e \in E(P_1)$, $\phi''(n_1) = 1$, and $\phi''(e) = \phi'(e)$  if $e \not \in E(P_1) \cup \{n_1\}$. The desired nowhere-zero 6-flow on $(G,\sigma)$ is
$(D,\phi)$ with $\phi(e) = \phi''(e) + 1$ if $e \in E(P_2)$, and $\phi(e) = \phi''(e)$ otherwise.
\end{proof}

\section{General case}

In this section we prove a general statement. 

\begin{theorem}\label{thm:general}
Let $(G,\sigma)$ be a flow-admissible signed cubic graph with $N_{\sigma}=\{uv, xy\}$, and let $G^*=(V(G),E(G)\cup\{ux\}-\{uv,xy\})$ be an unsigned graph. If $G^*$ admits a nowhere-zero $k$-flow for some integer $k$ such that $ux$ receives flow value 1, then $(G,\sigma)$ admits a nowhere-zero $(k+1)$-flow $(D,\phi)$ with the following properties: 
\begin{enumerate} 
\item $\phi(e)>0$, for every $e\in E(G)$, 
\item $\phi(uv)=\phi(xy)=1$, and
 \item there exists a $v$-$y$-path $P$ such that $\phi^{-1}(k)\subseteq E(P)$ and $\phi^{-1}(1)\cap E(P)=\emptyset$.
\end{enumerate}
\end{theorem}

\begin{proof}
Let $(D^*, \phi^*)$ be a nowhere-zero $k$-flow of $G^*$ as described in the statement, furthermore we may assume that $\phi^*(e)>0$ for every $e\in E(G^*)$. Suppose that $ux$ is oriented from $u$ to $x$, and let $P$ be a directed $y$-$v$-path of $G^*-\{ux\}$ (which exists by Lemma~\ref{lemma:directeduvpath} and the fact that $\phi^*(ux)=1$). We define $(D,\phi)$ on $(G,\sigma)$ as follows. For $e\in E(G)\cap E(G^*)$ we set $D(e)=D^*(e)$. Let $xy$ be extroverted and $uv$ be introverted, and let $\phi(xy)=\phi(uv)=1$. If $e\notin P$, then $\phi(e)=\phi^*(e)$, and if $e\in P$, then $\phi(e)=\phi^*(e)+1$. It is easy to see that $(D,\phi)$  is a desired nowhere-zero $(k+1)$-flow. 
\end{proof}

The previous theorem combined with the following observation provides several interesting corollaries. 

\begin{observation}\label{obser:G*}
Let $(G,\sigma)$ be a flow-admissible signed cubic graph with $N_{\sigma}=\{uv,xy\}$, and let $G^*=(V(G),E(G)\cup\{ux\}-\{uv,xy\})$ be an unsigned graph. If no 2-edge-cut of $(G,\sigma)$ contains a negative edge, then $G^*$ is flow-admissible.
\end{observation}

\begin{proof}
Suppose for a contradiction that $G^*$ is not flow-admissible. Then $G^*$ contains a bridge $b$. If $b=ux$, then either $(G,\sigma)$ has two components, each containing a negative edge, or $\{uv,xy\}$ is a 2-edge-cut of $(G,\sigma)$. In the first case $(G,\sigma)$ is not flow-admissible, and in the second case there is a 2-edge-cut of $(G,\sigma)$ containing a negative edge, a contradiction. If $b\neq ux$, then $u$ and $x$ belong to the same component $H$ of $G^*-b$. If $v$ and $y$ both belong to $H$, then $b$ is a bridge of $(G,\sigma)$ with an all-positive signed graph on one side. By Lemma~\ref{flow-admissible}, $(G,\sigma)$ is not flow-admissible, a contradiction. If neither $v$ nor $y$ belongs to $H$, then $\{uv, xy, b\}$ is a 3-edge-cut containing two negative edges. Hence by Corollary~\ref{cor:3edgecut}, $(G,\sigma)$ is not flow-admissible. Suppose, without loss of generality, that one of $v$ and $y$, say $v$, belongs to $H$. Then $\{uv, b\}$ is a 2-edge-cut of $(G,\sigma)$ containing a negative edge, a contradiction.
\end{proof}

We are ready to state the corollaries.

\begin{theorem}
If $(G,\sigma)$ is a flow-admissible signed cubic graph with $N_{\sigma}=\{uv,xy\}$, then $(G,\sigma)$ has a nowhere-zero 7-flow $(D,\phi)$ such that $\phi(uv)=\phi(xy)=1$, and all the edges with flow value 6 lie on a single path. 
\end{theorem}

\begin{proof}
Let $(G,\sigma)$ be a minimal counterexample to the theorem in terms of number of edges. By Theorem~\ref{thm:bridges}, $(G,\sigma)$ is 2-edge-connected. By Lemma~\ref{lemma:p2edgecut} and Theorem~\ref{6_Flow}, $N_{\sigma}$ does not form any 2-edge-cut. Assume that there is a 2-edge-cut $X$ containing one positive and one negative edge. Let $(G_1,\sigma_1)$ and all-positive $(G_2,\sigma_2)$ be resulting graphs of the 2-edge-cut reduction of $(G,\sigma)$ with respect to $X$. By Observation~\ref{obser:2red} $(G_1,\sigma_1)$ and $(G_2,\sigma_2)$ are flow-admissible. Furthermore, $(G_1,\sigma_1)$ has two negative edges and is smaller than $(G,\sigma)$. Therefore, $(G_1,\sigma_1)$ admits a nowhere-zero $7$-flow $(D_1,\phi_1)$ with the required properties. We may assume that $\phi_1(e)>0$, for every $e\in E(G_1)$. Let $f_i\in E(G_i)-E(G)$, for $i=1,2$. If $\phi_1(f_1)\leq 5$, then we use Corollary~\ref{6_Flow_1} to find a nowhere-zero $6$-flow $(D_2,\phi_2)$ on $(G_2,\sigma_2)$ with $\phi_2(f_2)=\phi_1(f_1)$. Otherwise, we find a nowhere-zero $6$-flow $(D_2,\phi_2)$ on $(G_2,\sigma_2)$ such that  $\phi_2(f_2)=5$ and $\phi_2(e)>0$, for every $e\in E(G_2)$. We modify $(D_2,\phi_2)$ into a nowhere-zero 7-flow by sending a flow value 1 along a directed circuit containing $f_2$ (note that by Lemma~\ref{lemma:directeduvpath}, there is a directed path between end-vertices of $f_2$).
In both cases, by Lemma~\ref{lemma:n2edgecut}, we can combine $(D_1,\phi_1)$ and $(D_2,\phi_2)$ into a desired nowhere-zero $7$-flow on $(G,\sigma)$, a contradiction. 

Finally, we may assume that every 2-edge-cut of $(G,\sigma)$ contains only positive edges. Let $G^*=(V(G),E(G)\cup\{ux\}-\{uv,xy\})$ be an unsigned graph obtained from $(G,\sigma)$. By Observation~\ref{obser:G*}, $G^*$ is flow-admissible, and by Corollary~\ref{6_Flow_1}, $G^*$ admits a nowhere-zero $6$-flow with flow value 1 on $ux$.  We obtain a contradiction by applying Theorem~\ref{thm:general}.
\end{proof}

The proof of the following corollary is very similar to the previous one, hence we omit it.

\begin{corollary}
If Tutte's 5-flow conjecture holds true, then Bouchet's conjecture holds true for all signed graphs with two negative edges. Moreover, for any bridgeless signed graph $(G,\sigma)$ with $N_{\sigma}=\{n_1,n_2\}$, there is a nowhere-zero $6$-flow $(D,\phi)$ with $\phi(e)>0$ for every $e\in E(G)$ such that $\phi(n_1)=\phi(n_2)=1$, and there is a path $P$ such that $\phi^{-1}(5)\subseteq E(P)$ and $\phi^{-1}(1)\cap E(P)=\emptyset$. 
\end{corollary}

A graph $G$ is \textit{cyclically $k$-edge-connected} if there exists no edge-cut $C$ with less than $k$ edges such that $G-C$ has two components that contain a circuit.
The \textit{oddness} $\omega(G)$ of a cubic graph $G$ is the minimum number of odd circuits of a 2-factor of $G$. In \cite{Steffen_2010} it is proven that if the cyclic connectivity of a cubic graph $G$ is at least $\frac{5}{2}\omega(G) - 3$, then $F(G,\texttt{1}) \leq 5$. Since for any nowhere-zero $k$-flow with $k\leq 5$, it is possible to choose a flow value on a particular edge, Theorem~\ref{thm:general} provides the following corollary.

\begin{corollary}
Let $(G,\sigma)$ be a flow-admissible signed cubic graph with $N_{\sigma}=\{uv,xy\}$, and let $G^*$ be an unsigned graph obtained from $(V(G),E(G)\cup\{ux\}-\{uv,xy\})$ by suppressing vertices of degree $2$. If $G^*$ is cyclically $k$-edge-connected and $k \geq \frac{5}{2}\omega(G') - 3$, then $F(G,\sigma) \leq 6$.
\end{corollary}
\medskip

\noindent
{\bf Acknowledgements.} The first author was supported by the project LO1506 of the Czech Ministry of Education, Youth and Sports and by the project
NEXLIZ --- CZ.1.07/2.3.00/30.0038, which is co-financed by the
European Social Fund and the state budget of the Czech Republic. The first author also gratefully acknowledges support from project 14-19503S of the Czech
Science Foundation.

\end{document}